\documentclass[11pt]{amsart}
\usepackage{amsmath,amsfonts,graphicx}
\usepackage{amssymb}
\usepackage{comment}
\usepackage{color}
\usepackage{tikz}
\usepackage{stmaryrd}
\usepackage{float}
\usepackage{stmaryrd}
\usepackage{multicol}
\usepackage[hidelinks]{hyperref}
\usepackage[nameinlink]{cleveref}
\usepackage{aliascnt}
\usepackage{enumitem}
\usepackage[caption = false]{subfig}
\usepackage[backend=biber, style=alphabetic, sorting=nyt, maxbibnames=99, maxalphanames=99]{biblatex}
\DeclareLabelalphaNameTemplate{
  \namepart[strwidth=1]{prefix}
  \namepart[strwidth=1]{family}
}

\newtheorem{thm}{Theorem}[section]
\newaliascnt{prpcnt}{thm}
\newtheorem{prp}[prpcnt]{Proposition}
\aliascntresetthe{prpcnt}
\crefname{prpcnt}{proposition}{propositions}
\Crefname{prpcnt}{Proposition}{Propositions}
\newtheorem{lmm}[thm]{Lemma}

\newtheorem{dfn}[thm]{Definition}
\newtheorem{rmk}[thm]{Remark}
\newtheorem{cjt}[thm]{Conjecture}

\newtheorem{qtn}[thm]{Question}

\crefname{prp}{proposition}{propositions}
\Crefname{prp}{Proposition}{Propositions}
\crefname{figure}{figure}{figures}
\Crefname{figure}{Figure}{Figures}

\renewcommand{\labelitemi}{$\blacktriangleright$}
\renewcommand{\labelitemii}{$\triangleright$}

\title{Marker endomorphisms for complex H\'enon maps}
\author{Ivan Cuezva} 
\address{\hspace{-3em}DER de math\'ematiques, ENS Paris-Saclay, Gif-sur-Yvette, France. Email: \texttt{ivan.cuezva@ens-paris-saclay.fr}} 

\begin{document}

\begin{abstract}
    This paper investigates some properties of the compound marker endomorphisms appearing in the study of the monodromy action of H\'enon maps. The main result is to present a computable criterion for checking the bijectivity of a single marker endomorphism. Some direct corollaries are also developed, keeping in mind the H\'enon dynamical aspect. Another point of the paper is the reformulation of compound marker endomorphisms in terms of permutive functions, which have been introduced by Hedlund. This reformulation allows to directly apply some of his results.
\end{abstract} 


\maketitle


\section{Introduction}

In this article, we investigate the monodromy problem of the complex H\'enon family :
\vspace{1em}
\begin{center}
    $H_{b,c} : \left\{
    \begin{array}{ccc}
    \mathbb{C}^2 &\to &\mathbb{C}^2  \\[1em]
    \begin{pmatrix}
    x \\
    y 
    \end{pmatrix} &\mapsto 
    &\begin{pmatrix}
    x^2 + c - by \\
    x 
    \end{pmatrix}
    \end{array}
    \right.$\\
\end{center}
\vspace{1em}
where $(b, c) \in \mathbb{C}^2$ is the parameter. This family has been introduced as the complexified dynamical system of the famous H\'enon family on $\mathbb{R}^2$, which exhibits a strange attractor. The monodromy problem tries to capture how the points in the Julia set are exchanged when the parameter changes along a loop in the parameter space. 

When $b\ne 0$, it is known that the H\'enon maps on their Julia sets are topologically conjugate to the shift map $\sigma : \Sigma_2\to\Sigma_2$ on the space of two-sided sequences on two symbols $\Sigma_2=\{0, 1\}^{\mathbb{Z}}$, for parameters in a certain locus $\mathcal{H}$ containing $\{(b, c)\in\mathbb{C}^{\times}\times \mathbb{C} : |c|>2(1+|b|)^2\}$, where $\mathbb{C}^{\times}=\mathbb{C}\setminus\{0\}$ (see \cite{Oberste-Vorth} and \cite{Ishii-Smillie}).  Moreover, we also know that the complex H\'enon map is hyperbolic on its Julia set for a parameter in $\mathcal{H}$ (see \cite{Ishii}). Thus, every point in the Julia set varies continuously when the parameter moves along a loop in $\mathcal{H}$, thanks to the hyperbolicity (see \cite{Shub}). This provides two codings of the Julia set by $\Sigma_2$ at the initial time and the final time of the loop. By composing these two coding maps, we obtain a shift commuting homeomorphism of $\Sigma_2$. Denote by $\mathrm{Aut}(\Sigma_2)$ the group of shift commuting homeomorphisms of $\Sigma_2$. By choosing a root point $(b_0, c_0) \in \mathcal{H}$ of loops, the above construction yields a group homomorphism $\rho : \pi_1 (\mathcal{H}, (b_0, c_0) ) \to \mathrm{Aut} (\Sigma_2)$. 

Note that when $b = 0$, the dynamics of $H_{b, c}$ reduces to the family of quadratic polynomials. In the seminal paper \cite{Blanchard}, a similar monodromy problem has been studied for polynomials of degree $d\geq 2$. They have shown that the corresponding monodromy homomorphism is surjective.

Even now, the image $\Gamma=\rho(\pi_1 (\mathcal{H}, (b_0, c_0)))$ of the monodromy homomorphism for the H\'enon family remains mysterious. Actually, even calculating the image of a given loop by $\rho$ is quite complicated. A conjecture due to Hubbard claims that $\rho$ is surjective, i.e. $\Gamma = \mathrm{Aut}(\Sigma_2)$. However, recent considerations about gyration numbers suggest that the shift map would not be contained in the image. Another conjecture due to Lipa \cite{Lipa} tries to give a method to calculate the monodromy action along a loop depending on its position in the parameter space. Aside from a method to calculate the monodromy of a specific loop, the conjecture also suggests a close relation between the monodromy action and a global combinatorial structure of the parameter space of the H\'enon family. As this conjecture makes use of certain automorphisms of $\Sigma_2$ called the compound marker automorphisms, a modern statement of Hubbard's conjecture is that $\Gamma$ is generated by these compound marker automorphisms.

The main result of this paper \Cref{mainthm} concerns the behaviour of these marker automorphisms, and consists of a criterion for a single marker endomorphism to be bijective. In a sense, the result obtained is that the single marker automorphisms are quite simple applications. Unfortunately, generalising this kind of criterion to compound marker endomorphisms turns out to be complicated. Another direction of our study stated in \Cref{secthm} is to use Hedlund's results~\cite{Hedlund} to investigate whether or not the compound marker automorphisms generate $\mathrm{Aut}(\Sigma_2)$.

This paper begins with some definitions and examples in \Cref{Definitions}. Then, \Cref{Markers} gives our criterion for a single marker endomorphism to be bijective, which turns out to be equivalent to be an involution. Using this result, \Cref{Consequences} gives some properties of the single marker endomorphisms. \Cref{Compound} explains why compound marker endomorphisms are much more difficult to study, and \Cref{Permutive} rephrase them in terms of permutive functions which have been studied by Hedlund. Finally, \Crefname{section}{Appendix}{Appendices}\Cref{Calculations}\Crefname{section}{Section}{Sections} explains an algorithm of Arai which can be used to calculate the monodromy action of a loop thanks to interval arithmetic. Some movies and pictures about visualization of the monodromy action can be found at \url{https://perso.crans.org/ivan/}.

\newpage

\section{Definitions and examples}
\label{Definitions}

This section introduces the main definitions and conjectures which will be of interest for us.

\begin{dfn}[Single marker endomorphism]
    Let $S$ be a finite word on $\{A,B,*\}$, containing exactly one $*$. The associated \emph{single marker endomorphism}, $\phi$, is constructed as follows: \\

    For $(c_i)_i \in \{A,B\}^{\mathbb{Z}}$, if for some slide $c_n,...,c_{n+j}$ matches $S$, where the $*$ can stand for a $A$ or a $B$, then $\phi((c_i)_i)$ is a copy of $(c_i)_i$ where the characters matching a $*$ are swapped. This operation is done at all the positions matching a $*$ at the same time.\\

     If the endomorphism is bijective, we call it a \emph{single marker automorphism}.
\end{dfn}

For example, if $S = B*BAA$ and $c = \overline{BABAA} = ...BABAABABAABABAA...$. Then, $\phi(c) = \overline{BBBAA}$. As the construction of the endomorphism only depends on the relative position of the letters, we obtain a continuous shift commuting function.

An example of a marker endomorphism being non bijective is $A*BAA$, which takes $\overline{BAA}$ to $\overline{BAB}$, and $\overline{BAB}$ to itself. Hence, it can't be injective. We will establish that $B*BAA$ is an automorphism later on.

In Lipa's conjecture, the following definition, generalising the single marker endomorphisms, appears:

\begin{dfn}[Compound marker endomorphism]
    Let $S_1, ..., S_k$ be some words on $\{A,B,*\}$, containing exactly one $*$. The associated \emph{compound marker endomorphism} $\phi$ is constructed as follows: \\

    For $(c_i)_i \in \{A,B\}^{\mathbb{Z}}$, if for some slide $c_n,...,c_{n+j}$ matches one of the $S_l$, where the $*$ can stand for a $A$ or a $B$, then $\phi((c_i)_i)$ is a copy of $(c_i)_i$ where the characters matching a $*$ are swapped. This operation is done at all the positions matching a $*$ at the same time.\\

     If the endomorphism is bijective, we call it a \emph{compound marker automorphism}.
\end{dfn}

For example, if $S_1 = A*BAA$, $S_2 = A*BABBA$, and $c = \overline{BAA}$, we have $\phi(c) = \overline{BAB}$, and $\phi^2 (c) = \overline{BAA}$.

The conjecture of Lipa~\cite{Lipa} tries to link the geometric position of a loop, expressed in terms of sheets and wakes, to its monodromy action. Roughly, the sheets are obtained by looking at the escape rate of the different critical points of the forward potential function. These critical points can be labelled in terms of finite sequences of $A$ and $B$. Thus, if for a value of $(b,c) \in \mathbb{C}^{\times} \times \mathbb{C}$, the $x$ critical point of the associated forward potential function does not escape, then $(b,c)$ is on the $x$ sheet. This gives a description of the parameter space as a superposition of different layers. At the same time, the wakes aims to divide the parameter space. Their use is inspired by the theory of external rays for the quadratic family. A wake $W$ gathers all the points having the same orbit portrait. By looking at the shorter angular sector we can also associate a period $p(W) \in \mathbb{N}^{\times}$, and a coding $K(W) \in \{A,B\}^{p(W)}$ to the wake $W$. A conspicuous sub-wake of $W$ is a wake $W'$ that surrounds $W$, is of higher degree, and such that no wake of period between $p(W)$ and $p(W')$ can be found between $W$ and $W'$. More details about sheets and wakes can be found in \cite{Richards}. This vocabulary allows us to formulate Lipa's conjecture as follows :

\begin{cjt}[Lipa's conjecture]
    Given a wake $W_1$ with conspicuous sub-wakes $W_2,...,W_n$, if $\gamma \in \pi_1(\mathcal{H}, (b_0, c_0))$ winds around the herd on the $x$ sheet of the wake $W_1$, then $\rho(\gamma)$ is the following compound marker automorphism :
    \begin{equation*}
        x * K(W_1), \text{ } x*K(W_2), \text{ } ... \text{ , } x*K(W_n)
    \end{equation*}
\end{cjt}

For example, the loop of \Cref{figure:loop} wraps around the herd on the $B$ sheet of the $BAA$ wake. Hence, the expected monodromy action is $B*BAA$, and the algorithm of \Crefname{section}{Appendix}{Appendices}\Cref{Calculations}\Crefname{section}{Section}{Sections} supports it. One remark is that Lipa's conjecture give the same monodromy action for a loop and for its inverse. Thus, we expect a lot of marker automorphisms to be involutions. This is one of the main motivations for \Cref{Markers}.

\begin{figure}[H]
\centering
\includegraphics[width=0.6\linewidth]{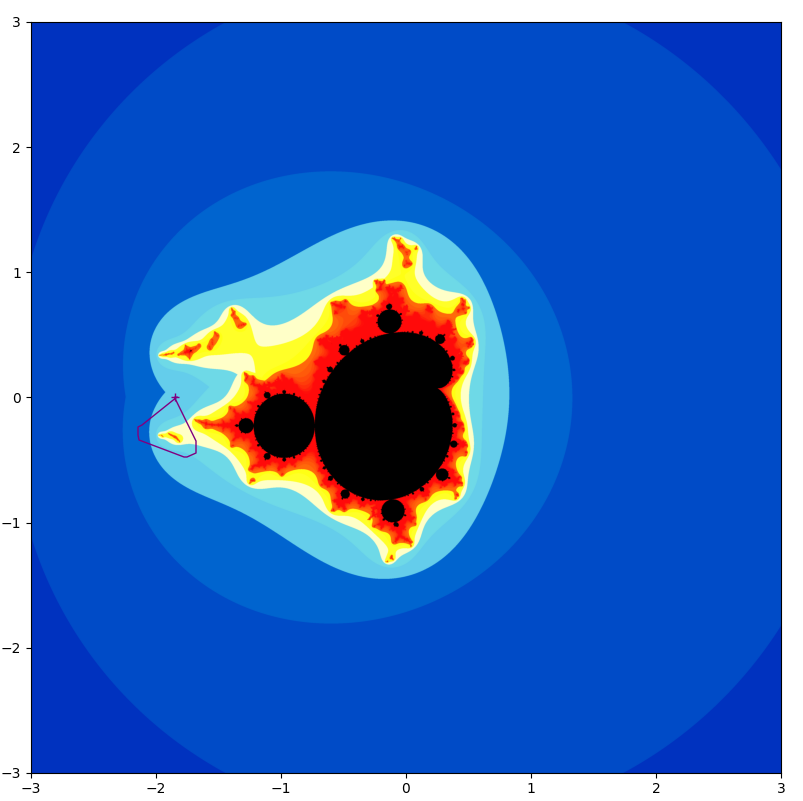}
\caption{Loop around an herd of the $B$ sheet and $BAA$ wake at $b = 0.15i$}
\label{figure:loop}
\end{figure}

Another question that arises in the study of the monodromy action is to determine the image $\Gamma = \rho(\pi_1(\mathcal{H},(b_0,c_0)))$. Initially, Hubbard suggested that $\Gamma = \mathrm{Aut} (\Sigma_2)$ in analogy with the surjectivity result in one dimension of \cite{Blanchard}. However, recent considerations suggest that the shift map may not belong to $\Gamma$. Hence, the following conjecture seems more realistic.

\begin{cjt}[Hubbard's conjecture]
Is $\mathrm{Aut} (\Sigma_2)$ generated by $\Gamma$ and the shift map $\sigma$?
\end{cjt}

A reasonable way to tackle Hubbard's conjecture could be to split the problem into the following two questions.
\begin{qtn}
Do all the compound marker automorphisms belong $\Gamma$ ?
\end{qtn}
This point motivates \Cref{Markers}, \Cref{Consequences} and \Cref{Compound}, but seems out of reach for now, as more work on Lipa's conjecture and compound marker endomorphisms would be useful. Assuming a positive answer to the first question, the second one, purely combinatorial, could be
\begin{qtn}
Is $\mathrm{Aut}(\Sigma_2)$ generated by the compound marker automorphisms and the shift ?
\end{qtn}
This point motivates \Cref{Permutive}, and seems more achievable.

\section{Saturated marker endomorphisms}
\label{Markers}

This section focuses on the study of single marker endomorphisms, as the compound case is more complicated. In particular, it introduces the notion of saturated marker endomorphism and proves that being saturated is equivalent to being bijective.\\

The following definition is motivated by the observation that when a marker endomorphism can swap two letters covered by a single word $S$, it has a very complicated behaviour. At first, this criterion was invented to check if a single marker endomorphism was an involution or not.

\begin{dfn}[Saturated marker endomorphism]
    Let $(\alpha_i)_{i \in \llbracket 1,n\rrbracket}$ and $(\omega_j)_{j \in \llbracket 1,m\rrbracket}$ be two finite words on $\{A,B\}$, and $\phi$ the single marker endomorphism generated by $\alpha * \omega$. We say that $\phi$ is \emph{saturated} if the two following points are true :

    \begin{itemize}
        \item For any $i \in \llbracket 1,n\rrbracket$, one of the following conditions is true :
        \begin{itemize}
            \item $\omega$ is not a prefix of $\alpha [i+1,n]*\omega$, where the $*$ can be replaced by $A$ or $B$.
            \item $\alpha[1,i-1]$ is not a suffix of $\alpha$.
        \end{itemize}
        \item For any $j \in \llbracket 1,m \rrbracket$, one of the following conditions is true :
        \begin{itemize}
            \item $\alpha$ is not a suffix of $\alpha * \omega [1,j-1]$, where $*$ can be replaced by $A$ or $B$.
            \item $\omega [j+1,m]$ is not a prefix of $\omega$.
        \end{itemize}
    \end{itemize}
\end{dfn}

\begin{rmk}
    The same notion is expressed in terms of overlaps in \cite{Hedlund} or \cite{Boyle}. It is used in a broader context to construct bijective marker endomorphisms, and some statements similar to the first implication of the following theorem can be found.
\end{rmk}

We use the convention that $\alpha[n+1,n]$ and $\alpha[1,0]$ are the empty words. The following theorem shows that single marker automorphisms are in fact involutions. This is quite reassuring regarding the first conjecture of Lipa. Moreover, as being saturated is easily checkable, the theorem is quite convenient.

\begin{thm}
    \label{mainthm}
    Let $(\alpha_i)_{i \in \llbracket 1,n\rrbracket}$ and $(\omega_j)_{j \in \llbracket 1,m\rrbracket}$ be two finite words on $\{A,B\}$, and $\phi$ the single marker endomorphism generated by $\alpha * \omega$. The following statements are equivalent :
    
    \begin{itemize}
        \item $\phi$ is saturated
        \item $\phi$ is an involution
        \item $\phi$ is bijective
    \end{itemize}
\end{thm}

\begin{proof}
    First, let us prove that if $\phi$ is saturated then it is an involution.
    
    We suppose that $\phi$ is saturated. Let us fix $c = (c_k)_{k \in \mathbb{Z}} \in \{A,B\}^\mathbb{Z}$, and $l \in \mathbb{Z}$.
    
    If $\phi(c)_l \neq c_l$, then by the definition of $\phi$, we have $c[l-n,l+m] = \alpha c_l \omega$. Thus, as $\phi$ is saturated, for $i \in  \llbracket 1,n\rrbracket$, $\alpha_i$ cannot be swapped by $\phi$, and for $j \in \llbracket 1,m \rrbracket$, $\omega_j$ cannot either. That gives $\phi(c)[l-n,l+m] = \alpha \tilde{c_l} \omega$, and $\phi^2(c)_l = c_l$, where $\tilde{c_l} = B$ if $c_l = A$, and the inverse either.

    If $\phi(c)_l = c_l$, and $\phi(c)[l-n,l+m] \neq \alpha c_l \omega$, we also have $\phi^2(c)_l = c_l$.

    Seeking for a contradiction, we suppose that $\phi(c)_l = c_l$, and $\phi(c)[l-n,l+m] = \alpha c_l \omega$. As $\phi(c)_l = c_l$, we know that $c[l-n,l+m] \neq \alpha c_l \omega$. Hence, one character of $c[l-n,l-1]$ or $c[l+1,l+m]$ has been exchanged by $\phi$. We denote is position by $s$. As $\phi(c)_s \neq c_s$, by saturation of $\phi$ we obtain $\phi(c)[s-n,s+m] = \alpha \tilde{c_s} \omega$. Moreover, as $\phi(c)[l-n,l+m] = \alpha c_l \omega$, if $s<l$, then $\alpha [1,n-(l-s)]$ is a suffix of $\alpha$, and $\omega$ is a prefix of $\alpha[n+2-(l-s),n]*\omega$. If $s>l$, $\omega[s-l+1,m]$ is a prefix of $\omega$, and $\alpha$ a suffix of $\alpha c_l \omega[1,s-l-1]$. In both cases we obtain a contradiction.
    
    As we can't have $\phi(c)_l = c_l$, and $\phi(c)[l-n,l+m] = \alpha c_l \omega$, and that in the other cases $\phi^2(c)_l = c_l$ for any $l \in \mathbb{Z}$, we obtained that $\phi$ is an involution.

    Finally, as the second implication is obvious, it remains to prove that being bijective implies being saturated. The two following lemmas combined give a proof of that by contraposition, which allows to concludes by circular implications.
\end{proof}

This lemma deals with an unsaturated character close enough to the $*$. The idea is that we can construct a word such that when applying $\phi$ we swap the character under the $*$ and the one under the unsaturated character which will then be blocking each other to go back to their original state. Thanks to that, we achieve to contradict the injectivity of $\phi$.

\begin{lmm}
    Let $\alpha \in \{A,B\}^n$ and $\omega \in \{A,B\}^m$, with $n,m \in \mathbb{N}$, and let $\phi$ be the endomorphism generated by $\alpha * \omega$.
    \begin{itemize}
        \item If there is $i \in \llbracket 1,n \rrbracket$ verifying $i \ge n-m +1$, and $\alpha _i$ not saturated, then $\phi$ is not bijective.
        
        \item If there is $j \in \llbracket 1,m \rrbracket$ verifying $j \le n$, and $\omega _i$ not saturated, then $\phi$ is not bijective.
    \end{itemize}
\end{lmm}

\begin{proof}
    We write $\alpha = \alpha_1 ... \alpha_n$ and $\omega = \omega_1 ... \omega_m$. For the first point, we take $i \in \llbracket 1,n \rrbracket$ verifying $i \ge n-m +1$, $\alpha _i$ not saturated, and maximal. Then, we construct $p \in \{A,B\}^{\mathbb{Z}}$ as follow. We start with $\alpha[1,n-i+1]\alpha h \omega$, with $h \in \{A,B\}$ such that $\omega$ is a prefix of $\alpha[i+1,n]h\omega$. Then, we add $\overline{\tilde{\alpha_1}}$ before and $\overline{\tilde{\omega_m}}$ at then end. That gives $p = \overline{\tilde{\alpha_1}} \alpha[1,n-i+1]\alpha h \omega \overline{\tilde{\omega_m}}$.

    The following shows that only the non saturated $\alpha_i$ and $h$ are being swapped by $\phi$. First, the consecutive $\tilde{\alpha_1}$ cannot be swapped as they cannot by preceded by $\alpha_1 ... \alpha_n$. The same goes for the $n$ first letters of $\alpha[1,n-i+1]\alpha$. Then, there is the non saturated $\alpha_i$ which is preceded by $\alpha$, and followed by $\omega$. The following $\alpha_{i+1}...\alpha_n$ cannot be swapped as $i$ has been taken maximal. Then, the $h$ is swapped as it is surrounded by $\alpha$ and $\omega$. For the tail, the presence of the consecutive $\tilde{\omega_n}$ make swapping impossible. We obtained that $\phi(p) = \overline{\alpha_1} \alpha[1,n-i+1]\alpha[1,i-1] \tilde{\alpha_i} \alpha[i+1,n] \tilde{h} \omega \overline{\omega_m}$.

    Now, as $i \ge n-m +1$, the $\tilde{h}$ impeach $\tilde{\alpha_i}$ to be swapped while $\tilde{\alpha_i}$ also blocks the $\tilde{h}$. Moreover, the $\overline{\alpha_1} \alpha[1,n-i+1]\alpha[1,i-1]$ and $ \omega \overline{\omega_m}$ parts are still saturated, and hence cannot be swapped. Thus, the orbit of $\phi$ can only swap $\alpha[i+1,n]$, and is thus finite. As $p \not\in \{\phi^k (p) \textbf{, } k \in \mathbb{N} \}$, it exists a minimal $k \in \mathbb{N}^{\times}$ such that $\phi^k (p) = \phi^l (p)$ with $l < k$ and $l \ge 1$, as the iterates of $p$ differ from $p$. If $l > 1$, by minimality of $k$, we have $\phi^{l-1} (p) \neq \phi^{k-1} (p)$, and $\phi^l (p) = \phi^k (p)$, $\phi$ is not injective. If $l = 1$, we also have $p \neq \phi^{k-1}(p)$ and $\phi(p) = \phi^k(p)$, and $\phi$ cannot be bijective. 

    For the second point, the reasoning is the same. We take $j \in \llbracket 1,m \rrbracket$ minimal verifying $j \le n$, and $\omega _j$ not saturated. Then, we consider the word $p = \overline{\tilde{\alpha_1}} \alpha h \omega \omega[m-j+1, m] \overline{\tilde{\omega _m}}$, where $h \in \{A,B\}$ allows to swap $\omega_j$. Again, as $j \le n$, the $\tilde{\omega_j}$ and the $\tilde{h}$ are blocking each other, and we have $\phi^k (p) = \phi^l(p)$, with $\phi^{k-1}(p) \neq \phi^{l-1}(p)$ for some $k,l \in \mathbb{N}^{\times}$. Thus, $\phi$ is not bijective.
\end{proof}

When the unsaturated character is far from the $*$, it gives some self similarity to the marker endomorphism. By making use of it, we achieve to use a technique similar to the one before.

\begin{lmm}
    \label{prop : far*}
    Given $\alpha \in \{A,B\}^n$ and $\omega \in \{A,B\}^m$, with $n,m \in \mathbb{N}$, and $\phi$ the endomorphism generated by $\alpha * \omega$.
    \begin{itemize}
        \item If there is $i \in \llbracket 1,n\rrbracket$, verifying $i < n-m+1$ and $\alpha_j$ not saturated, then $\phi$ is not bijective.
        \item If there is $j \in \llbracket 1,m\rrbracket$, verifying $j >n$ and $\omega_j$ not saturated, then $\phi$ is not bijective.
    \end{itemize}
\end{lmm}

\begin{proof}
    For the first point. We have $\alpha_i$ not saturated with $i < n-m+1$, then $\omega$ has to be a prefix of $\alpha[i+1,n]*\omega$, but $\alpha[i+1,n]$ has a length greater that $m$, and thus we can factorize $\alpha[i+1,n] = \omega \beta$ with $\beta \in \{A,B\}^{n-i-m}$. That gives $\alpha * \omega = \alpha[1,i-1] \alpha_i \omega \beta * \omega$. One problem that we could face is $\alpha[1,i-1]$ being having an unknown shape. With that in mind, we take $i \in \llbracket 1,n \rrbracket$ minimal verifying the hypothesis of the proposition. As $\alpha_i$ is not saturated, $\alpha[1,i-1]$ is a suffix of $\alpha$. But we have shown that $\alpha$ ends with $\omega \beta$. If $\alpha[1,i-1]$ is larger that $\omega\beta$, then its shape is $\mu\alpha_i \omega \beta$ with $\mu$ a finite word on $\{A,B\}$. But as $\alpha[1,i-1]$ is a suffix of $\alpha$, in $\mu\alpha_i \omega \beta$ the $\alpha_i$ cannot be saturated, which contradicts the minimality of $i$. Thus, $\alpha[1,i-1]$ is a suffix of $\omega\beta$. Now, we define $p = \overline{\alpha_i \omega \beta}$. As $p$ is periodic, by symmetry, its iterates under $\phi$ are also periodic with the same period. At the first iterate $\phi(p) = \overline{\tilde{\alpha_i} \hat{\omega} \hat{\beta}}$, where $\hat{\beta} \in \{A,B\}^{n-i-m}$ and $\hat{\omega} \in \{A,B\}^{m}$. Then, the $\tilde{\alpha_i}$ at the beginning of each pattern cannot be swapped as it is not preceded by $\alpha$, because of the previous $\tilde{\alpha_i}$. Thus, we have $\{ \phi^k (p) \text{, } k \in \mathbb{N}^{\times}\} \subset \{\overline{\tilde{\alpha_i} \hat{\omega} \hat{\beta}} \text{, } \hat{\omega} \in \{A,B\}^{m} \text{ and } \hat{\beta} \in \{A,B\}^{n-i-m} \}$. The last set being finite, it exists a minimal $k \in \mathbb{N}^{\times}$, $k > 1$ such that $\phi^k (p) = \phi^l (p)$ with $l < k$ and $l \ge 1$, as none iterate of p can start with a $\alpha _i$. By minimality of $k$, we have $\phi^{l-1} (p) \neq \phi^{k-1} (p)$, and as $\phi^l (p) = \phi^k (p)$, $\phi$ is not bijective.

    For the second point, the proof is the same. As $\omega_j$ is not saturated and $j>n$, then we can factorize $\omega = \beta \alpha \omega_j \gamma$, with $\beta \in \{A,B\}^{j-n-1}$, $\gamma \in \{A,B\}^{m-j}$. Taking $j \in \llbracket 1,m \rrbracket$ maximal with the required property, gives that $\gamma$ is a prefix of $\beta \alpha$. Then, we define $p = \overline{\beta \alpha \omega_j}$. In $\phi(p)$, the last $\omega_j$ are swapped, but cannot be swapped again as they are blocking each other. As the consecutive images of $p$ are periodic with same period, and the initial state $p$ cannot be reached again by iterating, it exists $k \in \mathbb{N}^{\times}$, such that $\phi(p)^k$ is the image of two distinct sequences. Thus $\phi$ cannot be injective.
\end{proof}

\section{Consequences for H\'enon maps}
\label{Consequences}

This section gathers some direct consequences of \Cref{mainthm}. This first result deals with the single marker endomorphisms with no tail or head.

\begin{prp}
    \label{prop:sw}
    Given a finite word $\omega \in  \{A,B\}^{m}$ for $m \in \mathbb{N}^{\times}$, the marker endomorphisms $*\omega$ and $\omega *$ are not bijective.
\end{prp}

\begin{proof}
    For $*\omega$ the last letter of $\omega$ is not saturated. For $\omega *$ the first one is not.
\end{proof}

The following result is more interesting. Using Lipa's conjectures vocabulary, it means that given a wake we can always find a sheet where a loop could generate an automorphism.

\begin{prp}
    \label{prop : complete}
    Given a finite word $\omega \in  \{A,B\}^{m}$ for $m \in \mathbb{N}^{\times}$, there exist $n \in \mathbb{N}$ and $\alpha \in  \{A,B\}^{n}$ such that $\alpha*\omega$ is an automorphism.
\end{prp}

\begin{proof}
If $\omega = \overline{A}^m$, we can consider $AB*\omega$ that is saturated. If $\omega = \overline{B}^m$, then $BA*\omega$ is not saturated either.

Now, we can suppose that $\omega$ is not $\overline{A}^m$ or $\overline{B}^m$.
    
Let us write $\omega = \omega_1 ... \omega_m$. The idea is to place the word $\tilde{\omega_1}\overline{\omega_1}^{m}$ at the left of the endomorphism in order to saturate the tail. We have to be careful doing that as we also need to have a saturated head. To do that we add an $\overline{\omega_1 \tilde{\omega_1}}$ before the $*$.

All of this gives the word $p = \tilde{\omega_1}\overline{\omega_1}^{m}\overline{\omega_1 \tilde{\omega_1}} * \omega$. The following proves that $\phi$ is a saturated marker endomorphism.

The first assumption made on $\omega$ gives that the first $\tilde{\omega_1}$ is saturated, as $\omega$ cannot be a prefix of $\overline{\omega_1}^{m}\overline{\omega_1 \tilde{\omega_1}}$. For the same reason the first $\omega_1$ of $\overline{\omega_1}^{m}$ is also saturated. Then, the other $\omega_1$ of $\overline{\omega_1}^{m}$ are saturated because they are not preceded by $\tilde{\omega_1}$. For the $\overline{\omega_1 \tilde{\omega_1}}$ part, $\omega_1$ is followed by a $\tilde{\omega_1}$ and hence saturated, and $\tilde{\omega_1}$ preceded by a $\omega_1$, which makes him also saturated. Finally, the tail is saturated because the first $\tilde{\omega_1}$ cannot be matched with the $m$ following $\omega_1$. The head is not a suffix of $p[1,m+3+j]$ for any $j \in \llbracket 1,m \rrbracket$.
\end{proof}

Conversely, taking $\alpha$ empty in the following result, we obtain that given a wake it exists a sheet where there is either no herd to turn around or no space for a loop to only capture this herd.

\begin{prp}
    Given $\alpha \in \{A,B\}^{n}$ and $\omega \in  \{A,B\}^{m}$ two finite words, with $n \in \mathbb{N}$ and $m \in \mathbb{N}$, there exist $k \in \mathbb{N}$ and $\gamma \in  \{A,B\}^{k}$ such that $\gamma\alpha*\omega$ is not an automorphism.
\end{prp}

\begin{proof}
    If $\omega$ is not empty we can take $\gamma = h\omega$, with $h \in \{A,B\}$ arbitrary. Then, in $h\omega\alpha*\omega$ the first letter $h$ is not saturated, which gives the result. If $\omega$ is empty we just need a non empty head to obtain a non bijective marker endomorphism.
\end{proof}

\begin{rmk}
    In the two previous results we fixed a wake $\omega$, but we could also fix a sheet $\alpha$ and obtain similar results on the tail.
\end{rmk}

In Lipa's conjectures the notion of kneading sequence appears to give a coding $K(W)$ to a wake $W$. However, as explained in \cite{Bruin}, all the words $\omega \in \{A,B\}$ are not able to form a kneading sequence. For example, the kneading sequence $BABBAA$ cannot appear. However, our precedent algorithm gives us that, $B*BABBAA$ is an automorphism. Thus, if this automorphism is generated by the monodromy action, it would have to arise from the concatenation of various loops. For that reason, it would be interesting to ask :

\begin{qtn} 
Is there a loop having $B \star BABBAA$ as monodromy action?
\end{qtn}

Finally, checking the saturation of a single marker endomorphism is quadratic on its size. Hence, we can obtain a list of the single marker automorphisms up to a certain size. These are exactly the marker automorphisms with head of size less than $2$ and tail of size less that $3$ :

\renewcommand{\labelitemi}{}

\begin{multicols}{5}
\begin{itemize}[label={}, leftmargin=0pt, itemindent=0pt, labelsep=0pt, topsep=0pt, itemsep=0pt, parsep=0pt]
    \item $A*BA$
    \item $A*BBA$
    \item $A*ABB$
    \item $B*AB$
    \item $B*BAA$
    \item $B*AAB$
\columnbreak
    \item $AA*BA$
    \item $AA*BAA$
    \item $AA*BBA$
    \item $AA*BAB$
    \item $AA*ABB$
    \item $BA*B$
\columnbreak
    \item $BA*BA$
    \item $BA*BB$
    \item $BA*BBA$
    \item $BA*AAB$
    \item $BA*ABB$
    \item $BA*BBB$
\columnbreak
    \item $AB*A$
    \item $AB*AA$
    \item $AB*AB$
    \item $AB*AAA$
    \item $AB*BAA$
    \item $AB*BBA$
\columnbreak
    \item $AB*AAB$
    \item $BB*AB$
    \item $BB*BAA$
    \item $BB*ABA$
    \item $BB*AAB$
    \item $BB*ABB$
\end{itemize}
\end{multicols}

\renewcommand{\labelitemi}{$\blacktriangleright$}
\renewcommand{\labelitemii}{$\triangleright$}

This exhaustive list of compound marker endomorphisms allows experimental considerations to know if $\Gamma$ contains all the single marker automorphisms. For now, in \cite{Arai}, and thanks to his algorithm explained in \Crefname{section}{Appendix}{Appendices}\Cref{Calculations}\Crefname{section}{Section}{Sections}, Arai proved that the following single marker endomorphisms are in $\Gamma$ : $A*BA$; $B*BAA$; $BA*BA$ and $AAB*BAA$. By conjugating with $*$, this also gives the following single marker automorphisms : $B*AB$; $A*ABB$, $AB*AB$ and $BBA*ABB$. However, with regard to the previous list, this is still unsatisfactory.

\section{Compound marker endomorphisms}
\label{Compound}

This short section explains why the notion of saturated marker endomorphism do not extend easily to the compound case.\\

At first sight, taking the composition of two marker endomorphisms can seem a bit similar to taking their compound marker endomorphism. However, it allows some interactions between marker endomorphisms which makes the situation difficult. For example, Lipa proved in his thesis \cite{Lipa} that the marker $A*BAA, A*BABBA$ is an involution even if it is made from two non bijective markers. Conversely, Lipa also proved that the marker $A*BAA, A*BABBA, B*BBA$ is not bijective, while $B*BAA$ and $A*BAA, A*BABBA$ are.

A simple generalisation of saturation could be to require that each letter of each word cannot be surrounded by another word of the compound marker endomorphism. However, the example of the marker automorphism $A*BAA, A*BABBA$, shows that this criterion is not sufficient. A possible approach would be to require that for any unsaturated letter, there is another string where this letter is swapped, as in $A*BAA, A*BABBA$ where the last $A$ of $A*BAA$ and the second $B$ of $A*BABBA$ are unsaturated. Even so, the situation remains complicated as some letters can be added after the unsaturated character, and may not be saturated.

Anyway, having a criterion for involutive marker automorphisms is not enough, as \[ABBABA*ABBAB, ABBAB*BABBAB, ABBAB*AABBAB\] is of order $4$. Its construction is inspired by Theorem 6.13 of \cite{Hedlund} which embeds finite groups in the automorphisms of the shift.

Thanks to \cite{Hedlund} and \cite{Arai}, we also know the existence of compound marker endomorphisms of infinite order. For example, the composition of $*$ and $AAB*BAA$, which can be expressed as the compound marker automorphism made up of all the $\alpha * \omega$ with $\alpha , \omega \in \{A,B\}^3$, and $(\alpha , \omega) \neq (AAB,BAA)$, is one of them. Moreover, using the algorithm explained in \Crefname{section}{Appendix}{Appendices}\Cref{Calculations}\Crefname{section}{Section}{Sections} \cite{Arai} provided a loop with such monodromy.

\section{Permutive functions}
\label{Permutive}

This section links the notion of compound marker endomorphisms appearing in Lipa's conjectures to the tools developed by Hedlund in \cite{Hedlund}.\\

For $n \in \mathbb{N}$, we denote the set of functions from $\{A,B\}^{2n+1}$ to $\{A,B\}$ by $F(n)$.

\begin{dfn}
    For $f \in F(n)$, we define
    \begin{center}
    $f_\infty : \left\{
    \begin{array}{ccc}
    \Sigma_2 &\to &\Sigma_2  \\[1em]
     (x_k)_{k \in \mathbb{Z}} &\mapsto & (f(x_{k-n} ... x_{k+n}))_{k \in \mathbb{Z}}
    \end{array}
    \right.$\\
\end{center}
\end{dfn}

\begin{rmk}
    The original definition of \cite{Hedlund} only takes into account the slide $x_{k} ... x_{k+n}$. However, up to compose with the shift map, the definitions are equivalent. 
\end{rmk}

Theorem 3.1 of \cite{Hedlund} proves that $f_\infty$ is shift commuting and continuous. We denote by $F_\infty (n)$ the set of functions obtained this way, and $F_\infty = \bigcup\limits_{n \in \mathbb{N}} F_\infty (n)$. Theorem 3.4 of \cite{Hedlund} states that any shift commuting continuous function of $\Sigma_2$ is in $F_\infty$. Theorem 5.13 of \cite{Hedlund} also states that if a shift commuting continuous function is injective then it is bijective. Thus $\mathrm{Aut} (\Sigma_2)$ is the subset of injective functions of $F_\infty$.

In order to make the connection between compound marker endomorphisms and $F_\infty$, we need the following definition from \cite{Hedlund}.

\begin{dfn}[Permutive function]
    Given $f \in F(n)$, and $i \in \llbracket -n,n \rrbracket$, we say that $f$ is \emph{permutive} in $x_i$ if for any $x_{-n},...,x_{i-1},x_{i+1},...,x_{n} \in \{A,B\}$, the function from $\Sigma_2$ to $\Sigma_2$ defined by $g(s) = f(x_{-n},...,x_{i-1},s,x_{i+1},...,x_{n})$ is bijective.
\end{dfn}

We denote by $P_\infty$ the subset of $F_\infty$ consisting of permutive functions in $x_0$.

\begin{thm}\label{secthm}
The set of compound marker endomorphisms is $P_\infty$.
\end{thm}

\begin{proof}
    First, given $\alpha_1 * \omega_1,...,\alpha_k * \omega_k$ a compound marker endomorphism, we define $n = \max\limits_{i \in \llbracket 1,k \rrbracket} |\alpha_i |, |\omega_i|$. Then, we define $f \in F(n)$ as follows. If for some $i \in \llbracket 1,k \rrbracket$, $s[n-|\alpha_i|+1,n] = \alpha_i$ and $s[n+2,n+1+|\omega_i|] = \omega_i$, then $f(s) = \tilde{s_{n+1}}$. Else $f(s) = s_{n+1}$. With this definition, $f_\infty$ is the compound marker endomorphism generated by $\alpha_1 * \omega_1,...,\alpha_k * \omega_k$, and $f$ is permutive in $x_0$.

    Conversely, given $f \in F(n)$ permutive in $x_{0}$, we will construct a corresponding compound marker endomorphism. We start with an empty list $l$, and for each $\alpha, \omega \in \{A,B\}^{n}$, if $s \mapsto f(\alpha,s,\omega)$ is the non trivial permutation of $\{A,B\}$, we add $\alpha * \omega$ to $l$. Thus, the compound marker endomorphism arising from $l$ is $f$.
\end{proof}

\begin{rmk}
    Theorem 6.9 of \cite{Hedlund} and the previous corresponding between permutive functions and compound marker endomorphisms allows to generalize \Cref{prop:sw} as follows. For any $\omega_1,...,\omega_k$ finite words on $\{A,B\}$, the compound marker endomorphisms $*\omega_1,...,*\omega_k$ and $\omega_1*,...,\omega_k *$ are bijective if and only if they are equal to the compound marker endomorphism $*$.
\end{rmk}

\begin{rmk}
    If $f \in F(n)$ is permutive in an other variable than $x_{0}$ then $f$ can be expressed as the composition of a compound marker endomorphism and a power of the shift.
\end{rmk}

Thanks to Hedlund's results, to know if the compound marker automorphisms and the shift are generating $\mathrm{Aut} (\Sigma_2)$, we only need to study the functions $f : \{A,B\}^n \to \{A,B\}$ which are not permutive in any of their variables, and such that $f_\infty$ is injective. This study can be summarized by the following two questions :

\begin{itemize}
    \item Is there $n \in \mathbb{N}$ and $f \in F(n)$ not permutive in any of its variables, with $f_{\infty}$ injective ? \vspace{1em}
    
    \item If there is such $f \in F(n)$, can we express it as the composition or the inverse of compound maker endomorphisms ?
\end{itemize}

\vspace{2em}

\textbf{Acknowledgements.}
This paper comes from an internship under the supervision of Yutaka Ishii, at the university of Kyushu, in Japan. I am deeply grateful to him for his advices and corrections. I would also like to thank Toshiki Usui who encouraged me to study a bijectivity criterion for single marker endomorphisms and who, together with Thomas Richards, answered many of my questions.

\printbibliography

\appendix
\renewcommand{\thesection}{\Alph{section}}
\renewcommand{\thesubsection}{\Alph{section}.\arabic{subsection}}
\renewcommand{\thesubsubsection}{\Alph{section}.\arabic{subsection}.\arabic{subsubsection}}

\newpage

\section{Computer calculations}
\label{Calculations}

This appendix explains the algorithm of Arai, from \cite{Arai}, used to calculate the monodromy action of specific loops. We also used it to check the first conjecture of Lipa for the loop surrounding the herd on the $B$ sheet of the $BAA$ wake of \Cref{figure:loop}. As the automorphisms studied are commuting with the shift, we only need to study one fixed position in the coding of the points. The goal is to understand when this digit changes.\\

A crucial observation due to \cite{Hubbard} is that for $b,c$ fixed, the Julia set $J(H_{b,c})$ is trapped in the four dimensional cube of size $\beta = \frac{1}{2} (|b|+ 2+ \sqrt{(|b|+2)^2 + 4|c|})$. Thus, we can chose $\epsilon > 0$ small and subdivide $[-\beta,\beta]^4$ is small cubes of side $\epsilon$. Then, as the H\'enon map and its inverse are defined using only products and sums, we can use interval arithmetic to calculate strict bounds for the forward and backward images of a small cube of size $\epsilon$. As the Julia set is invariant by both $H_{b,c}$ and its inverse, if one of the small cubes escapes the cube of side $\beta$ after some iterations, then it cannot contain the Julia set. By iterating a lot of times both backward and forward, on cubes small enough, we obtain a quite sharp approximation of the Julia set.

\begin{figure}[H]
\centering
\includegraphics[width=0.8\linewidth]{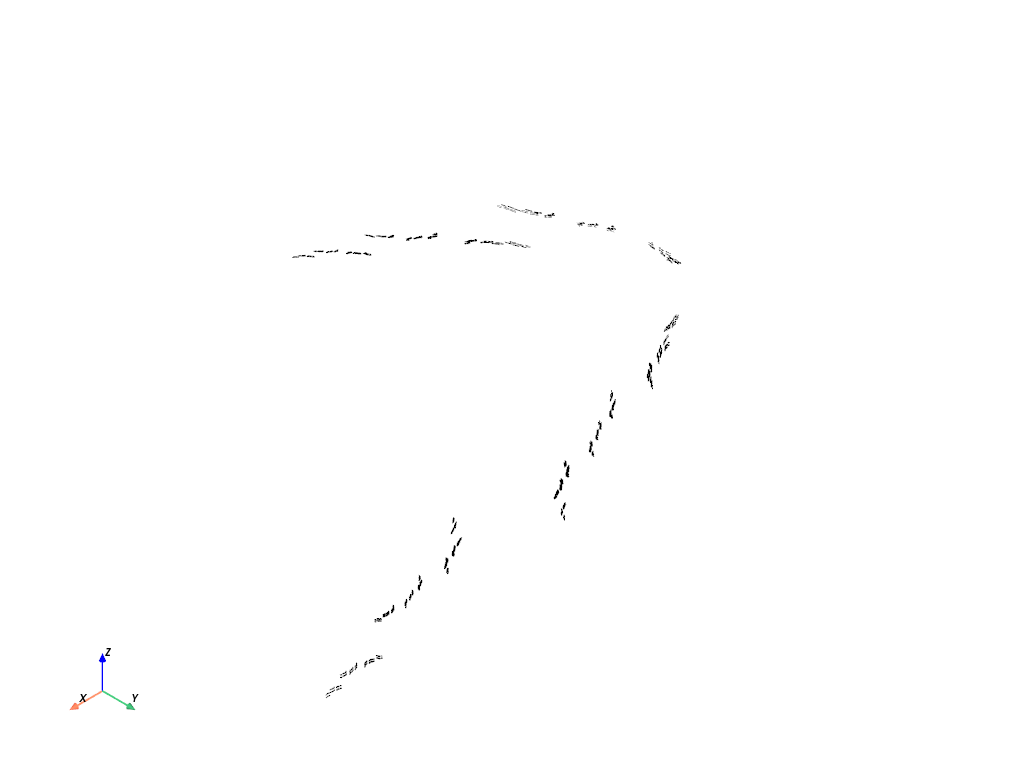}
\caption{Approximation of $J(H_{b,c})$ projected by dropping $\Im (y)$}
\label{figure:julia}
\end{figure}

Moreover, if we know how the Markov partition is constructed, and $\epsilon$ is small enough, we can obtain some digits of the coding of each cube, depending where are his forward and backward images. In \Cref{figure:julia}, for example, the Markov partition is done by labelling $A$ the points where $\Re (y) > 0$, and $B$ the points with $\Re (y) < 0$.

At this point we use interval arithmetic at a fixed value for $b,c$. However, we can also trap a loop in the parameter space in small cubes of size $\eta$, and use interval arithmetic both for $b,c$ and $x,y$. Doing this allows to calculate an exact approximation of the Julia sets for all the parameter values in the size $\eta$ boxes.

Now, to calculate the monodromy action, what we would like to do is to start at a point where we know the Markow partition and then spread the coding around the loop. To do that we need to suppose that the loop is contained in the hyperbolic locus. Thus, if we fix a code $x \in \Sigma_2$, and take a small cube of values for $b,c$, the set of points coded by $x$ in all the $J(H_{b,c})$ is connected. Thanks to hyperbolicity, we know that if a point of the Julia set is labelled by $A$ then it is in the same connected component as an initial point labelled by $A$, thus if we can divide the cubes in two sets at a positive distance, one having all the initial cubes coded $A$ and the other the initial cubes coded $B$, we can spread the coding of the initial cubes to their whole connected component. By iterating, if $\epsilon$ and $\eta$ are small enough, we achieve to propagate the Markov partition along the loop. Then, it only remain to compare with the initial partition to see what blocks moved. After that, we can obtain a more precise coding of the cubes that moved and of those who did not.

By supposing the hyperbolicity around the loop of \Cref{figure:loop}, running the algorithm for $\epsilon = \eta = 2^{-10}$ gives that the first digit is swapped if and only if it is surrounded by $B$ and $BAA$. Thus, the monodromy action is the single marker automorphism $B*BAA$. In his paper \cite{Arai}, Arai was able to check the hyperbolicity, and calculate the monodromy action for other loops in different slices of the parameter space. He obtained the single marker automorphisms $A*BA$; $B*BAA$; $BA*BA$ and $AAB*BAA$.

\end{document}